\documentclass[11pt]{article}
\usepackage[reqno]{amsmath}
\usepackage{amssymb, enumerate, amsthm}
\usepackage{graphicx}
\usepackage{array,url}
\usepackage{bbm}  

\usepackage{fullpage}

\newtheorem{theorem}{Theorem}[section]
\newtheorem{lemma}[theorem]{Lemma}
\newtheorem{conjecture}[theorem]{Conjecture}
\newtheorem{corollary}[theorem]{Corollary}
\newtheorem{proposition}[theorem]{Proposition}

\newcommand\la{\lambda}

\newcommand\cD{{\mathcal D}}
\newcommand\cI{{\mathcal I}}
\newcommand\cS{{\mathcal S}}

\newcommand\CC{{\mathbb C}}

\newcommand\RR{{\mathbb R}}

\DeclareMathOperator\tr{tr}

\newcommand\indeg{d^-}
\newcommand\outdeg{d^+}
\newcommand\innbd[2]{N^-_{#1}(#2)}
\newcommand\outnbd[2]{N^+_{#1}(#2)}
\newcommand\Cabc[1]{\vec{C}_3(#1)}

\title{Hermitian adjacency spectrum and switching equivalence of mixed graphs}

\author{Bojan Mohar\thanks{Supported in part by an NSERC Discovery Grant (Canada), by the Canada Research Chair program, and by the Research Grant of ARRS (Slovenia).}~~\thanks{On leave from: IMFM, Department of Mathematics, Jadranska 19, Ljubljana, Slovenia.} \\[1mm]
{Department of Mathematics}\\{Simon Fraser University}\\{Burnaby, B.C. V5A 1S6}\\ \texttt{mohar@sfu.ca}
}

\begin{document}
\maketitle

\begin{abstract}
  It is shown that an undirected graph $G$ is cospectral with the Hermitian adjacency matrix of a mixed graph $D$ obtained from a subgraph $H$ of $G$ by orienting some of its edges if and only if $H=G$ and $D$ is obtained from $G$ by a four-way switching operation; if $G$ is connected, this happens if and only if $\la_1(G)=\la_1(D)$.
  All mixed graphs of rank 2 are determined and this is used to classify which mixed graphs of rank 2 are cospectral with respect to their Hermitian adjacency matrix. Several families of mixed graphs are found that are determined by their Hermitian spectrum in the sense that they are cospectral precisely to those mixed graphs that are switching equivalent to them.
\end{abstract}


\section{Introduction}
\label{sec:intro}

A \emph{mixed graph} is obtained from an undirected graph by orienting a subset of its edges. Formally, a mixed graph $D$ is given by its \emph{vertex-set\/} $V=V(D)$, the set $E_0(D)$ of \emph{undirected edges} and the set $E_1(D)$ of \emph{directed edges} or \emph{arcs}. We write $E(D)=E_0(D)\cup E_1(D)$ and distinguish undirected edges as unordered pairs $\{x,y\}$ of vertices, while the arc are ordered pairs $(x,y)$, $x,y\in V$. Instead of $\{x,y\} \in E_0(X)$ or $(x,y) \in E_1(X)$, we write $xy$ for short when it is clear or unimportant whether the edge $xy$ is undirected or directed. Note that for the arc $xy = (x,y)$, the first vertex $x$ is its initial vertex, and $y$ is its terminal vertex. In this paper we do not allow multiple edge between the same pair of vertices (unless stated otherwise); in this case $uv$ is an undirected edge if and only if both $uv$ and $vu$ belong to $E(D)$.

The \emph{Hermitian adjacency matrix} of a mixed graph $D$ is a matrix $H=H(D)\in \CC^{V\times V}$, whose $(u,v)$-entry $H_{uv}$ is the imaginary unit $i$ if there is an arc from $u$ to $v$, $-i$ if there is an arc from $v$ to $u$, $1$ if both arcs exist or $uv$ is an undirected edge, and $0$ otherwise. Since digons (two arcs joining two vertices in opposite direction) behave like undirected edges, we may assume that we have no digons. This matrix was introduced by Liu and Li \cite{LiuLi15} in the study of graph energy and independently by Guo and the author \cite{GuoMo}. The latter paper contains a thorough introduction to the properties of the Hermitian adjacency matrix, displays several results about its eigenvalues, and discusses similarities and differences from the case of undirected graphs.

The Hermitian adjacency matrix $H=H(D)$ is Hermitian and thus all of its eigenvalues are real. We will denote by $\la_j(D)$ the $j$th largest eigenvalue of $H$ (multiplicities counted), so that $\la_1(D)\ge \la_2(D)\ge \cdots \ge \la_n(D)$, where $n=|V(D)|$. The multiset of the eigenvalues is called the \emph{$H$-spectrum} of $D$, and the eigenvalues are referred to as the \emph{$H$-eigenvalues} when we want to stress that they come from the Hermitian adjacency matrix. Two mixed graphs with the same $H$-spectrum are said to be \emph{$H$-cospectral} (or just \emph{cospectral} for short).

Problems about existence of cospectral (undirected) graphs have long history. The original question goes back to 1956 when G\"unthard and Primas \cite{GuPr56} raised the question whether there are non-isomorphic cospectral trees. They believed that no such trees exist. The seminal paper by Collatz and Sinogowitz \cite{CoSi57} provided the first examples of cospectral trees. Later, Schwenk \cite{Schw73} proved that ``almost all trees are cospectral'', a result that was explored further in \cite{McK77} and \cite{GoMcK82}. The original motivation came from mathematical chemistry. In 1966, additional motivation came from mathematical physics and differential geometry when Kac \cite{Ka66} asked his famous question: ``Can one hear the shape of a drum?'' Of course, the meaning of this question is about which geometric (or combinatorial) properties of Riemannian manifolds (and graphs) are determined by the spectrum of its Laplacian (or adjacency) operator.

There are many results about cospectral graphs and about those that are determined by their spectrum. Two surveys on this subject have appeared quite recently \cite{vDHa03,vDHa09}. We refer to these for additional information.

Guo and Mohar \cite{GuoMo} discussed operations that preserve the $H$-spectrum of mixed graphs. They found an operation, called \emph{four-way switching}, which can be described as a special similarity transformation (see Section \ref{sect:4WS}). This operation does not occur when dealing with undirected graphs, but in case of mixed graphs, it gives a plethora of cospectral mixed graphs with the same underlying graph. In this paper we explore this aspect a bit further and ask when it is the case that every mixed graph that is $H$-cospectral with $D$ has the same underlying graph and can be obtained from $D$ by using a four-way switching, possibly followed by the reversal of all directed edges (which also preserves the spectrum). In such a case we say that $D$ is \emph{determined by its $H$-spectrum} (shortly \emph{DHS}).

The main results of this paper are twofold.
First, it is shown that an undirected graph $G$ is cospectral with the Hermitian adjacency matrix of a mixed graph $D$ obtained from a subgraph $H$ of $G$ by orienting some of its edges if and only if $H=G$ and $D$ is obtained from $G$ by a four-way switching operation. Furthermore, if $G$ is connected, this happens if and only if $\la_1(G)=\la_1(D)$. In addition it is also proved when the spectrum of $D$ is the same as the negative of the spectrum of $G$. See Theorems \ref{thm:cospectral to orientation} and \ref{thm:anticospectral}.

Second, all mixed graphs of rank 2 are determined and this is used to classify which mixed graphs of rank 2 are cospectral with respect to their Hermitian adjacency matrix. See Theorem \ref{thm:rank 2}. This result has some consequences on mixed graphs that are determined by their $H$-spectra. The paper ends with a beautiful characterization showing which complete multipartite digraphs with parts of size $n-a$, $n$, and $n+a$ (and all arcs directed ''clockwise'') are DHS; see Corollary \ref{cor:big DHS family}.

\section{Basic definitions and notation}
\label{sec:defn}

Let $D$ be a mixed graph. If $E_1(D)=\emptyset$, then we  say that $D$ is \emph{undirected} or, more simply, that $D$ is a \emph{graph}.
The \emph{underlying graph of a mixed graph} $D$, denoted by $G(D)$, is the graph with vertex-set $V(D)$ and edge-set $E = \{ \{x,y\} \mid xy \in E(D)\}$. If a mixed graph $D$ has no undirected edges and no digons, we say that $D$ is an \emph{oriented graph}. The set of all mixed graphs whose underlying graph is $G$ will be denoted by $\cD(G)$; its elements will be referred to as \emph{mixed graphs based on $G$}.
In this paper, we restrict our attention to \emph{simple} mixed graphs without digons, meaning that $G(D)$ is a simple graph and each edge of $G(D)$ is either in $E_0(D)$ or precisely one of its orientations is in $E_1(D)$. In some cases we will allow multiple edges, in which case we shall speak of (\emph{mixed\/}) \emph{multigraphs}. However, digons are always excluded.

For a vertex $x \in V(D)$, we define the set of \emph{in-neighbours} of $x$ as $\innbd{D}{x} = \{u \in V(D) \mid ux \in E(D)\}$ and the set of \emph{out-neighbours} as $\outnbd{D}{x} = \{u \in V(D) \mid xu \in E(D)\}$.
The \emph{in-degree} of $x$, denoted $\indeg(x)$, is the number of in-neighbours of $x$. The \emph{out-degree} of $x$, denoted $\outdeg(x)$, is the number of out-neighbours of $x$. Note that undirected edges count toward in- and out-degrees.

For a mixed graph $D$ with vertex-set $V= V(D)$ and edge-set $E = E(D)$, we consider the \emph{Hermitian adjacency matrix} $H = H(D) \in \CC^{V\times V}$, whose entries $H_{uv}$ are given by
\[
H_{uv} = \begin{cases} ~1 & \text{if } uv \text{ and } vu \in E ; \\
~i & \text{if } uv \in E \text{ and } vu \notin E ; \\
-i & \text{if } uv \notin E \text{ and } vu \in E ; \\
~0 & \text{otherwise.}\end{cases}
\]
In the case of mixed multigraphs (without digons), the entries of $H$ are obtained by summing up contributions of all edges joining two vertices.

Observe that $H$ is a Hermitian matrix and so is diagonalizable with real eigenvalues.

\section{Four-way switching}
\label{sect:4WS}

The \emph{converse} of a mixed graph $D$ is the mixed graph $D^T$ with the same vertex-set, same set of undirected edges and the arc-set $E_1(D^T) = \{xy \mid yx \in E_1(D) \}$.
It is immediate from the definition of the Hermitian adjacency matrix that if $D$ is a mixed graph and $D^T$ is its converse, then $H(D^T) = H(D)^T = \overline{H(D)}$. This implies the following result (see \cite{GuoMo}).

\begin{proposition}
\label{prop:converse cospectral}
A mixed graph $D$ and its converse $D^T$ are $H$-cospectral.
\end{proposition}

Guo and Mohar \cite{GuoMo} unveiled a more complicated transformation which preserves the $H$-spectrum. Suppose that the vertex-set of $D$ is partitioned in four (possibly empty) sets, $V(D) = V_1\cup V_{-1}\cup V_i\cup V_{-i}$. An edge $xy\in E(D)$ is said to be of \emph{type} $(j,k)$ for $j,k\in\{\pm1, \pm i\}$ if $x\in V_j$ and $y\in V_k$. The partition is \emph{admissible} if the following conditions hold:
\begin{itemize}
  \setlength\itemsep{0pt}
  \item[(a)] There are no digons of types $(1,-1)$ or $(i,-i)$.
  \item[(b)] All edges of types $(1,i),(i,-1),(-1,-i),(-i,1)$ are contained in digons.
\end{itemize}
A \emph{four-way switching} with respect to a partition $V(D) = V_1\cup V_{-1}\cup V_i\cup V_{-i}$ is the operation of changing $D$ into the mixed graph $D'$ by making the following changes:
\begin{itemize}
  \setlength\itemsep{0pt}
  \item[(c)] reversing the direction of all arcs of types $(1,-1),(-1,1),(i,-i),(-i,i)$;
  \item[(d)] replacing each digon of type $(1,i)$ with a single arc directed from $V_1$ to $V_i$ and replacing each digon of type $(-1,-i)$ with a single arc directed from $V_{-1}$ to $V_{-i}$;
  \item[(e)] replacing each digon of type $(1,-i)$ with a single arc directed from $V_{-i}$ to $V_1$ and replacing each digon of type $(-1,i)$ with a single arc directed from $V_i$ to $V_{-1}$;
  \item[(f)] replacing each non-digon of type $(1,-i),(-1,i),(i,1)$ or $(-i,-1)$ with the digon.
\end{itemize}

\begin{figure}
\centering
\includegraphics[scale=0.9]{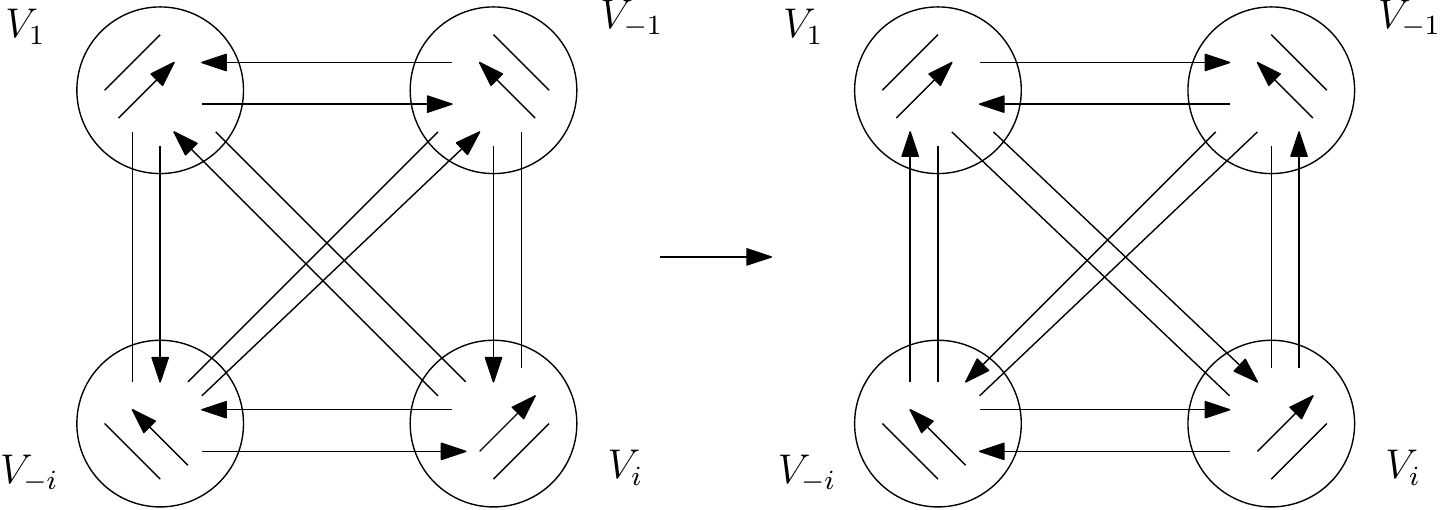}
\caption{Four-way switching on the admissible edges.}
\label{fig:four-way-switching}
\end{figure}

The following result was proved in \cite{GuoMo}.

\begin{theorem}
\label{thm:generalized switch}
If a partition $V(D) = V_1\cup V_{-1}\cup V_i\cup V_{-i}$ is admissible, then the four-way switching gives a mixed graph that is cospectral with $D$.
\end{theorem}

Let $D'$ be the mixed graph obtained from $D$ by a four-way switching (with respect to an admissible partition). The proof in \cite{GuoMo} shows that its Hermitian adjacency matrix $H(D')$ is similar to $H(D)$. The corresponding
similarity transformation uses the diagonal matrix $S$ whose entry $S_{vv}$ is equal to $j\in\{\pm1, \pm i\}$ if $v\in V_j$. The entries of the matrix $H' = S^{-1}HS$ are given by the formula
$$
    H'_{uv} = H_{uv} S_{vv}/S_{uu}.
$$
It is clear that $H'$ is Hermitian and that its non-zero elements are in $\{\pm1, \pm i\}$. The entries within the parts of the partition remain unchanged. Undirected edges of type $(1,-1)$ would give rise to the entries $-1$, but since these are excluded for admissible partitions, this does not happen. On the other hand, any other entry in this part is multiplied by $-1$, and thus arcs of types $(1,-1)$ or $(-1,1)$ just reverse their orientation. Similar conclusions are made for other types of edges. Admissibility is needed in order that $H'$ has no entries equal to $-1$.

We can use the similarity interpretation discussed above as a definition. Let $\cI=\{\pm1, \pm i\}$ be the multiplicative group generated by the fourth primitive root of unity $i$ and let $\cS$ be the set of all diagonal matrices whose diagonal elements belong to $\cI$. For a mixed graph $D$, each $S\in\cS$ defines a partition $V(D) = V_1\cup V_{-1}\cup V_i\cup V_{-i}$, where $V_j = \{v\in V(D) \mid S_{vv} = j\}$, $j\in\cI$. If this partition is admissible, then we say that the mixed graph $D'$ with Hermitian adjacency matrix $H(D') = S^{-1}H(D)S$ is obtained from $D$ by a four-way switching. We say that mixed graphs $D_1$ and $D_2$ are \emph {switching equivalent} if one can be obtained from the other by a sequence of four-way switchings and operations of taking the converse.

\begin{proposition}
\label{prop:equivalence}
For each graph $G$, switching equivalence gives an equivalence relation on the set $\cD(G)$.
\end{proposition}

\begin{proof}
Clearly, if two mixed graphs are switching equivalent, they have the same underlying graph $G$. The similarity relation with matrices in $\cS$ is symmetric and transitive since $\cS$ is a group for matrix multiplication. To see what happens if we take the converse after a similarity, we first note that $H^T = \overline H$ (complex conjugation) and that $\cS$ is closed for complex conjugation. If $D'$ is obtained from $D$ by a similarity and then taking the converse, its matrix satisfies:
$$
    H(D') = \overline{S^{-1}H(D)S} = \overline{S^{-1}}~\, \overline{H(D)}~\, \overline{S} = S\,\overline{H(D)}\,S^{-1}.
$$
This means that $D'$ is obtained by a four-way switching from the converse of $D$. This now easily implies the result.
\end{proof}

The proof of the proposition shows that the switching equivalence class of $D$ contains all mixed graphs that are obtained from $D$ or from $D^T$ by a single application of a four-way switching.

There are some special cases of the 4-way switching in which two of the four sets are empty. They involve two of the sets (one of which can be assumed to be $V_1$) and one of the other three possible sets. The cut between them has only two types of edges (see Figure \ref{fig:four-way-switching}), and the edges in the cut are transformed according to the rule in the 4-way switching. We call these \emph{$2$-way switching} operations. There are really two different cases:

(a) A \emph{directed $2$-way switching} reverses the edges in a cut consisting of arcs only.

(b) A \emph{mixed $2$-way switching} has a cut containing undirected edges and arcs in one direction only. The switching replaces the arcs in the cut by undirected edges, and replaces all formerly undirected edges by arcs in the direction opposite to the direction of former arcs.

Based on the above, a mixed graph $D$ is said to be \emph{determined by its Hermitian spectrum} (\emph{DHS\/}) if every mixed graph with the same Hermitian spectrum can be obtained from $D$ by a four-way switching, possibly followed by the reversal of all directed edges.


The following result was proved in \cite{LiuLi15} (see also \cite{GuoMo}).

\begin{theorem}
\label{thm:forests}
Let\/ $F$ be a forest. Then all mixed graphs whose underlying graph is isomorphic to $F$ are cospectral with $F$.
\end{theorem}

This result is easily obtained by using Theorem \ref{thm:generalized switch}, as changing an undirected cutedge to an arc or changing an arc to its reverse is a special case of a mixed 2-way switching.

\begin{figure}
\centering
\includegraphics[scale=1.3]{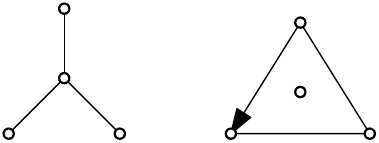}
\caption{Cospectral mixed graphs.}
\label{fig:DS but not DHS}
\end{figure}

An undirected graph is said to be \emph{determined by its spectrum} (DS) if every graph cospectral to it is isomorphic to it (see \cite{vDHa03,vDHa09}).
This is a weaker property than being DHS. Namely, if a graph is DHS, it is necessarily DS. The converse is not true as evidenced by two graphs in Figure \ref{fig:DS but not DHS} (the undirected star $K_{1,3}$ is DS).

The following graphs are well-known to be DS:
\begin{itemize}
\setlength\itemsep{0pt}
\item[\rm (a)]
Disjoint union of complete graphs (see \cite[Proposition 1 and Proposition 6]{vDHa03}).
\item[\rm (b)]
The balanced complete bipartite graph $K_{n,n}$ (see \cite[Proposition 5]{vDHa03}).
\item[\rm (c)]
Disjoint union of paths (see \cite[remark after Proposition 7]{vDHa03}).
\end{itemize}

For complete graphs, it was proved in \cite{GuoMo} that $K_n$ ($n\ge1$) is DHS. However, the results in Section \ref{sect:rank2} show that $K_{n,n}$ is not DHS whenever $n$ is not a square-free integer (and possibly never except for a few small values of $n$).




\section{Mixed graphs cospectral to their underlying graph}
\label{sect:cospectral to G(D)}

In this section it is characterized when a mixed graph can be cospectral (or antispectral -- see the definition in the sequel) to its underlying graph.

\begin{theorem}
\label{thm:cospectral to orientation}
Let $G$ be a connected undirected multigraph (multiple edges and loops allowed). Let $D$ be a spanning mixed multigraph obtained from $G$ by deleting some edges and orienting some of the non-loop edges so that no digons are created. The following statements are equivalent:
\begin{itemize}
\setlength\itemsep{0pt}
\item[\rm (a)]
$G$ and $D$ are $H$-cospectral.
\item[\rm (b)]
$\la_1(G) = \la_1(D)$.
\item[\rm (c)]
None of the edges have been deleted, i.e., $G(D)=G$, and
there exists a partition of the vertex-set of $D$ into four (possibly empty) parts $V_{1}$, $V_{-1}$, $V_{i}$, and $V_{-i}$ such that the following holds.
For $j \in \cI$, the mixed graph induced by $V_j$ in $D$ contains only undirected edges. Every other edge $uv$ of $D$ is such that $u \in V_j$ and $v \in V_{-ij}$ for some $j \in \cI$. See Figure \ref{fig:rhodelta}(a).
\item[\rm (d)]
$G$ and $D$ are switching equivalent.
\end{itemize}
\end{theorem}

\begin{figure}
\centering
\includegraphics[scale=1.1]{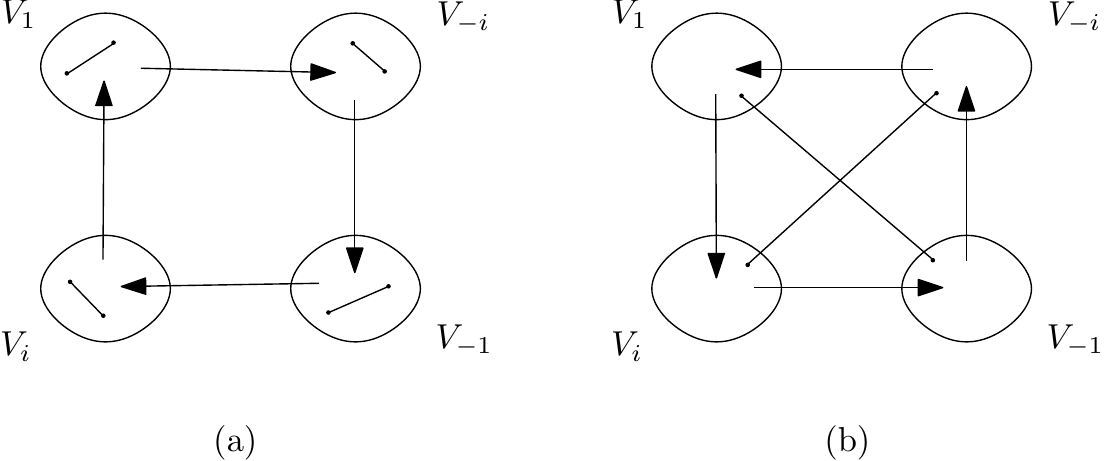}
\caption{Structure from Theorems \ref{thm:cospectral to orientation}(c) and \ref{thm:anticospectral}(c).}
\label{fig:rhodelta}
\end{figure}

\begin{proof}
It is obvious that (a) implies (b). It is also easy to see that (c) implies (d): if $D$ has the described structure, we first take its converse $D^T$ and then observe that the four-way switching with respect to the given partition yields $G$. Proposition \ref{prop:converse cospectral} and Theorem \ref{thm:generalized switch} show that (d) implies (a).

It remains to see that (b) implies (c). Assuming (b) holds, let $H=H(G)$, $H'=H(D)$, let $x$ be a normalized Peron vector for $H$. This means that $Hx=\la_1(G)x$, $x^THx=\la_1(G)$ and $x^Tx=1$. We may assume that $x$ is real and positive, and under this assumption it is uniquely determined.
Similarly, let $y\in \CC^V$ be a normalized eigenvector of $H'$ for $\la_1(D)$, and let $z\in \RR^V$ be defined by $z_v=|y_v|$, $v\in V$. Now,
\begin{eqnarray*}
\la_1(G) &=& y^*H'y = \sum_{u\in V}\sum_{v\in V} H'_{uv} \overline{y_u} y_v \\
      &\le& \sum_{u\in V}\sum_{v\in V} |H'_{uv}| z_u z_v \\
      &\le& \sum_{u\in V}\sum_{v\in V} H_{uv} z_u z_v \\
      &\le& \sum_{u\in V}\sum_{v\in V} H_{uv} x_u x_v = \la_1(D).
\end{eqnarray*}
Since $\la_1(G)=\la_1(D)$, we see that all inequalities must be equalities. The last inequality shows that $z=x$, since $x$ is the unique non-negative vector that attains the maximum of the quadratic form. It follows that $z_u>0$ for every $u$. Now, the middle inequality shows that $H_{uv}=|H'_{uv}|$ for all $u,v\in V$, which means no edges have been removed.  Finally, the first inequality shows that $H'_{uv} \overline{y_u} y_v = |H'_{uv}|z_u z_v = |H'_{uv}| |y_u| |y_v|$ for every edge $uv$. As we may assume that one of the vertices has $y_v\in \RR^+$, we conclude that all neighbors $u$ of that vertex have $y_u/|y_u| \in \{1, \pm i\}$. Since $G$ is connected, the repeated use of this argument shows that $y_u/|y_u|\in \cI$ for each $u\in V$. Let $V_j=\{u\in V\mid y_u/|y_u|=j\}$. This defines the partition, and it is straightforward to check that the edges within and between the parts are as claimed in (c).
\end{proof}

A special case of Theorem \ref{thm:cospectral to orientation} where $G$ is the complete graph $K_n$ was proved in \cite{GuoMo}. In the same paper, the special case of complete graphs was also done for the case of the next theorem.

We say that two mixed graphs $D$ and $D'$ are \emph{$H$-antispectral} if for each $H$-eigenvalue $\la$ of $D$ of multiplicity $k$, $-\la$ is an eigenvalue of $D'$ with the same multiplicity $k$.

\begin{theorem}
\label{thm:anticospectral}
Let $G$ be a connected undirected multigraph (multiple edges and loops allowed). Let $D$ be a spanning mixed multigraph obtained from $G$ by deleting some edges and orienting some of the non-loop edges so that no digons are created. The following statements are equivalent:
\begin{itemize}
\setlength\itemsep{0pt}
\item[\rm (a)]
$G$ and $D$ are $H$-antispectral.
\item[\rm (b)]
$\la_1(G) = -\la_n(D)$.
\item[\rm (c)]
None of the edges have been deleted, i.e., $D(G)=G$, and
there exists a partition of the vertex-set of $D$ into four (possibly empty) parts $V_{1}$, $V_{-1}$, $V_{i}$, and $V_{-i}$ such that the following holds.
For $j \in \cI$, the graph induced by $V_j$ in $G$ is an independent set. For each $j \in \cI$, every arc with one end in $V_j$ and one end in $V_{-j}$ is contained in a digon. Every other arc $uv$ of $D$ is such that $u \in V_j$ and $v \in V_{ij}$ for some $j \in \cI$.
See Figure \ref{fig:rhodelta}(b).
\end{itemize}
\end{theorem}

\begin{proof}
The proof is essentially the same as that of the previous theorem. We borrow the notation and note the difference that $y$ is an eigenvector of $\la_n(D)$. Again, the nontrivial part is to show that (b) implies (c).

Now we get the same chain of inequalities as in the previous proof except that we start with $|\la_n(D)|$ and that the final of the three conclusions is that
$-H'_{uv} \overline{y_u} y_v = |H'_{uv}|z_u z_v = |H'_{uv}| |y_u| |y_v|$. Again, we conclude by connectivity that $y_u/|y_u| \in \cI$ for each $u\in V$ and define $V_j=\{u\in V\mid y_u/|y_u| = j\}$. This gives the partition, and it is straightforward to check that the edges within and between the parts are as claimed in (c).
\end{proof}

\section{Mixed graphs with $H$-rank 2}
\label{sect:rank2}

In this section we classify all mixed graphs whose Hermitian adjacency matrix has rank 2. This result shows precisely which mixed graphs of rank 2 are cospectral and hence can be used to find some mixed graphs that are DHS.

It is easy to see that an undirected graph has rank 2 if and only if it consists of a complete bipartite graph $K_{a,b}$ ($1\le a\le b$) together with $t\ge0$ isolated vertices.

When we speak about the \emph{$H$-rank\/} of a mixed graph, we mean the rank of its Hermitian adjacency matrix. The following consequence of the basic properties of the rank of matrices will be used throughout.

\begin{lemma}
\label{lem:rank subgraph}
Suppose that $D$ is a mixed graph and $D'$ is an induced mixed subgraph of $D$. Then the $H$-rank of $D$ is greater or equal to the $H$-rank of $D'$.
\end{lemma}

From now on we shall concentrate on mixed graphs with $H$-rank 2.

\begin{lemma}
\label{lem:rank2 basic properties}
Suppose that $D$ is a mixed graph whose Hermitian adjacency matrix has rank $2$. Then $D$ has the following properties:
\begin{itemize}
\setlength\itemsep{0pt}
\item[\rm (a)]
$D$ consists of one connected component with more than one vertex together with some isolated vertices.
\item[\rm (b)]
Every induced subgraph of $D$ has $H$-rank $0$ or $2$.
\item[\rm (c)]
The underlying undirected graph $G(D)$ contains no induced path on at least $4$ vertices and no induced cycle of length at least $5$.
\end{itemize}
\end{lemma}

\begin{proof}
Since the trace of $H(D)$ is 0, $D$ has one negative and one positive eigenvalue and they have the same magnitude. This holds for each connected component, so we immediately see that (a) holds. Note that, in particular, no mixed graph has rank equal to 1. Thus, Lemma \ref{lem:rank subgraph} implies (b). In order to prove (c), it suffices to exclude the path $P_4$ since longer paths and long induced cycles contain induced $P_4$.

Suppose that $G(D)$ has an induced path on four vertices. By Theorem \ref{thm:forests}, every mixed graph whose underlying graph is $P_4$ has the same $H$-spectrum as $P_4$, which has rank 4. By Lemma \ref{lem:rank subgraph}, $D$ has $H$-rank at least 4, a contradiction.
\end{proof}

Let us now classify connected mixed graphs of order 3 and 4 whose $H$-rank is 2. We will say that a mixed graph whose underlying graph is the 3-cycle $C_3$ is an \emph{odd triangle} if it has an odd number of arcs.

\begin{lemma}
\label{lem:order 3 rank 2}
A connected mixed graph of order $3$ has rank $2$ if and only if either $G(D)$ is a path, or $D$ is an odd triangle.
\end{lemma}

\begin{proof}
By Theorem \ref{thm:forests}, all mixed graphs with $G(D)$ isomorphic to $P_3$ are cospectral to $P_3$ and thus of rank 2. Since $D$ is connected we may thus assume that $G(D)$ is a triangle. It is easy to see (cf. \cite{GuoMo}) that odd triangles have $H$-rank 2 and that even triangles have rank~3.
\end{proof}

For order 4 we will need a special $4\times 4$ matrix.

\begin{lemma}
\label{lem:special matrix rank 2}
For $x,y,z\in\{0,\pm i\}$, let $H(x,y,z)$ be the following matrix:
$$
   H(x,y,z) = \left[
                \begin{array}{cccc}
                  0 & 1 & 1 & 1 \\
                  1 & 0 & x & -y \\
                  1 & -x & 0 & z \\
                  1 & y & -z & 0 \\
                \end{array}
              \right]
$$
The rank of $H$ is equal to $2$ if and only if either $x=y=z=0$, or one of $x,y,z$ is equal to $0$, another one to $i$ and the remaining one to $-i$.
\end{lemma}

\begin{proof}
It is easy to compute that the characteristic polynomial of $H$ is
$$
   \phi(H,t) = t^4 + (x^2+y^2+z^2-3)t^2 - (x+y+z)^2.
$$
Thus, the rank is 2 if and only if $x+y+z=0$. One solution to this is when all of them are 0. If they are all $\pm i$, their sum is an odd multiple of $i$, thus nonzero. So, one of them must be 0. Then the other two sum to 0, meaning one is the negative of the other. This completes the proof.
\end{proof}

Any mixed graph based on $K_{1,3}$ has rank 2.
Figure \ref{fig:C4 rank 2} shows all mixed graphs of $H$-rank 2 whose underlying graph is $K_{2,2}$ (the 4-cycle), and Figure \ref{fig:K4-rank 2} shows all mixed graphs of rank 2 that are obtained from the graph $K_4^-$ ($K_4$ minus an edge). In either of the three possibilities, they are all switching equivalent to each other. This gives all connected mixed graphs of $H$-rank 2 on four vertices.

\begin{figure}
\centering
\includegraphics[scale=1.0]{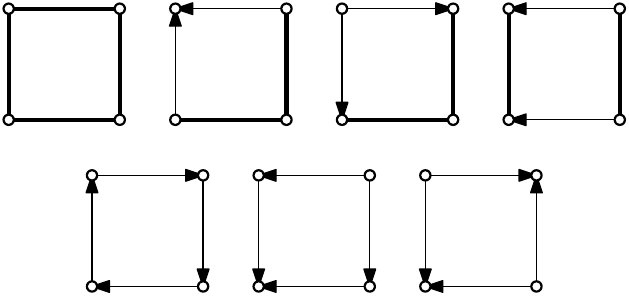}
\caption{Mixed graphs of $H$-rank 2 based on $C_4$.}

\label{fig:C4 rank 2}
\end{figure}

\begin{figure}
\centering
\includegraphics[scale=1.2]{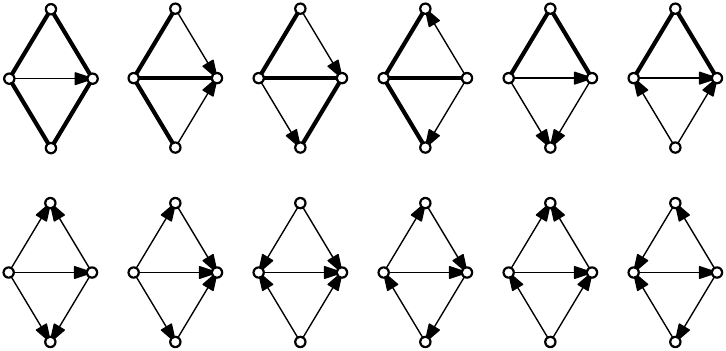}
\caption{Mixed graphs of $H$-rank 2 based on $K_4^-$.}
\label{fig:K4-rank 2}
\end{figure}

\begin{lemma}
\label{lem:order 4 rank 2}
Let $D$ be a connected mixed graph of order $4$. Then $D$ has $H$-rank $2$ if and only if one of the following holds:
\begin{itemize}
\setlength\itemsep{0pt}
\item[\rm (a)]
$G(D)$ is isomorphic to $K_{1,3}$.
\item[\rm (b)]
$G(D)$ is isomorphic to $C_4$ and $D$ is one of the mixed graphs shown in Figure \ref{fig:C4 rank 2}.
\item[\rm (c)]
$G(D)$ is isomorphic to $K_4^-$ and $D$ is one of the mixed graphs shown in Figure \ref{fig:K4-rank 2}.
\end{itemize}

\end{lemma}

\begin{proof}
Suppose first that $G(D)$ has no triangles. Then it is bipartite. If it is not $P_4$ (which gives rank 4), it is either $K_{1,3}$ or $K_{2,2}$. Each mixed graph based on $K_{1,3}$ is cospectral to $K_{1,3}$ (Theorem \ref{thm:forests}), so it is of rank 2. If $D$ is based on $K_{2,2} = C_4$, then it is of rank 2 if and only if it is switching equivalent to $C_4$, which means it is one of the mixed graphs shown in Figure \ref{fig:C4 rank 2}. There are two other switching classes but their rank is bigger than 2 (see \cite{GuoMo}).

We may henceforth assume that $G(D)$ contains a triangle. By Lemma \ref{lem:order 3 rank 2}, this must be an odd triangle, call it $T$. The fourth vertex $v$ is adjacent to one of the vertices of $T$. If $v$ has only one neighbor in $T$, then it is easy to see that $H(D)$ has rank at least 3. (The rows of $H(D)$ corresponding to the vertices not adjacent to $v$ are linearly independent but they cannot generate the row of the third vertex which is adjacent to $v$.) If $v$ is adjacent to two or three vertices in $T$, then $G(D)$ is isomorphic to $K_4^-$ or to $K_4$.

Suppose first that $D$ has a vertex which is incident with three undirected edges. Since every triangle in $D$ is odd, the edges between the neighbors of this vertex are all directed. Thus, $H(D)$ is equal to the matrix $H$ in Lemma \ref{lem:special matrix rank 2}, where each of $x,y,z$ is 0 or $\pm i$. By the lemma, we see that one of the three, say $z$ must be 0 and that $x=-y$. Since $G(D)\ne K_{1,3}$, this means that $x\ne0$ and thus $G(D)=K_4^-$. Moreover, since $x=-y$, the two arcs among the neighbors of $v$ are both incoming or both outgoing from one of the neighbors to the other two (see the second and the fourth graph in Figure \ref{fig:K4-rank 2}).

Let $uv$ be an edge of $D$, where $u$ and $v$ both have degree 3 in $G(D)$. Let $w,z$ be the other two vertices, and let $H_v,H_u,H_w,H_z$ be the corresponding rows of $H(D)$. Suppose first that $uv$ is undirected. Assuming that $D$ has $H$-rank 2, the triangle $uvw$ is odd, and we may assume that the edge $uw$ is undirected and $vw$ or $wv$ is directed. By passing to the converse mixed graph if necessary, we may assume that $vw$ is an arc in $D$. Excluding the case from the previous paragraph, the edge $vz$ is undirected. If the edge $wz$ is present, then the triangle $vwz$ is odd and hence this edge is undirected. In any case ($wz$ present or not), $z$ is incident with only one directed edge. By a mixed 2-way switching we can change the directed edge into an undirected one. After this, $u$ becomes incident with three undirected edges, and we are done by the previous paragraph.

From now on we may assume that every edge joining two vertices of degree 3 in $G(D)$ is directed in $D$. Let $uv$ be such an arc. If $u$ has no incoming arcs, then we can make $uv$ undirected by a mixed 2-way switching. Similarly, if $v$ has no outgoing arcs.
If $G(D)=K_4$, then all edges are directed and by the above, each vertex has positive indegree and outdegree. Clearly, there is a vertex $u$ whose outdegree is 1. If $v$ is its outneighbor, we can switch at $v$, exchanging its in- and outgoing edges. This makes $u$ having indegree 3, and we are done since we can switch to a previously treated case. The same can be done if $G(D)=K_4^-$ and all edges of $D$ are directed.

It remains to treat the case when $G(D)=K_4^-$, $uv$ is an arc, and there is an undirected edge. We may assume that $wu$ is an incoming arc to $u$. Since there is an undirected edge and all triangles are odd, the edges $uz$ and $zv$ are undirected and $vw$ (or $wv$) is an arc. But then we can switch at $z$ and obtain the previous situation where every edge is directed. This completes the proof.
\end{proof}

\begin{lemma}
\label{lem:complete bi or tripartite rank 2}
If $D$ is a connected mixed graph of\/ $H$-rank $2$, then $G(D)$ is either a complete bipartite or a complete tripartite graph.
\end{lemma}

\begin{proof}
Let $G=G(D)$. Suppose that $G$ has no triangles. Then $G$ is bipartite since otherwise, a shortest odd cycle in $G$ would be induced and would contradict Lemma \ref{lem:rank2 basic properties}(c). A shortest path between any two nonadjacent vertices in opposite parts of the bipartition would induce a path on at least 4 vertices. Since $G$ has no induced $P_4$, there are no such nonadjacent vertices. Since it contains at least one edge, it is necessarily a complete bipartite graph.

Suppose now that $G$ contains a triangle $T=uvw$. Let $Q$ be the largest induced complete tripartite subgraph of $G$ containing $T$. If $Q\ne G$, then there is a vertex $z\in V(G)\setminus V(T)$ that is adjacent to $Q$. We may assume that it is adjacent to $T$. By Lemma \ref{lem:order 4 rank 2}, $z$ is adjacent to precisely two vertices in $T$, say to $u$ and $v$. Let $A,B,C$ be the classes of $Q$, where $u\in A$, $v\in B$, and $w\in C$. Considering any triangle $uvw'$, $w'\in C$, we see by Lemma \ref{lem:order 4 rank 2} that $z$ is not adjacent to $w'$. Considering all triangles $u'vw$ and $uv'w$ ($u'\in A$, $v'\in B$), we see that $z$ is adjacent to every $u'\in A$ and $v'\in B$. But then we can add $z$ to $Q$ to obtain a larger complete tripartite graph, contradicting that $Q$ was maximal possible.
\end{proof}

Two vertices $u,v$ in $D$ are \emph{twins} if $D$ is switching equivalent to a mixed graph $D'$ in which $u$ and $v$ have exactly the same neighborhood. More precisely, for each $w\in V(D')$, we have $vw\in E(D')$ if and only if $uw\in E(D')$, and $wv\in E(D')$ if and only if $wu\in E(D')$. By removing or adding twins the rank of the Hermitian adjacency matrix remains the same (but the $H$-spectrum changes). The relation of being a twin of each other is an equivalence relation on $V(D)$. Let $[v]$ denote the equivalence class containing the vertex $v$. It is easy to see that $D$ is switching equivalent to a mixed graph $D'$ such that for each $u,v\in V(D')$ and for any $u'\in[u]$, $v'\in[v]$, we have $u'v'\in E(D')$ if and only if $uv\in E(D')$. This enables us to define the quotient mixed graph $T(D)$ whose vertices are the equivalence classes, $V(T(D))=\{[v]\mid v\in V(D)\}$, and $[u] [v]\in E(T(D))$ if $uv\in E(D')$. Note that $D'$ is determined only up to switching equivalence, and thus also $T(D)$ is determined only up to switching equivalence. We say that $T(D)$ has been obtained from $D$ by \emph{twin reduction}. The following observation, whose proof is left to the reader, enables us to assume that there are no twins when classifying mixed graphs of a fixed rank.

\begin{lemma}
\label{lem:twin reduction}
Let $D_1$ and $D_2$ be mixed graphs with the same underlying graph. Then they are switching equivalent if and only if $T(D_1)$ and $T(D_2)$ are switching equivalent.
\end{lemma}

Let $a\le b\le c$ be positive integers. We will denote by $\vec{C}_3(a,b,c)$ the complete tripartite mixed graph with parts $A$, $B$, $C$, where $|A|=a$, $|B|=b$ and $|C|=c$, and with all arcs from $A$ to $B$, all arcs from $B$ to $C$, and all arcs from $C$ to $A$. See Figure \ref{fig:cospectralrank2} for an example. Note that all vertices in $A$, $B$ and $C$ (respectively) are twins, and thus the twin reduction gives a mixed graph isomorphic to the directed 3-cycle $\vec{C}_3$.

\begin{figure}
\centering
\includegraphics[scale=2.0]{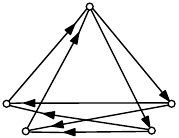}
\caption{$\Cabc{1,2,2}$}
\label{fig:cospectralrank2}
\end{figure}

\begin{theorem}
\label{thm:rank 2}
Let $D$ be a mixed graph of order $n$ whose $H$-rank is equal to $2$ and let $\rho$ be its positive eigenvalue. Then $D$ is switching equivalent either to $K_{a,b}\cup tK_1$ or to $\vec{C}_3(a,b,c)\cup tK_1$, where $t\ge 0$. In the former case we have
\begin{equation}
\label{eq:1}
  n=a+b+t \quad \text{and} \quad \rho^2=ab,
\end{equation}
and in the latter case
\begin{equation}
\label{eq:2}
  n=a+b+c+t \quad \text{and} \quad \rho^2=ab+bc+ac.
\end{equation}
Conversely, for any $t\ge0$ and $1\le a\le b\le c$, the mixed graphs $K_{a,b}\cup tK_1$ and $\vec{C}_3(a,b,c)\cup tK_1$ have $H$-rank $2$ and they satisfy (\ref{eq:1}) and (\ref{eq:2}), respectively.
\end{theorem}

\begin{proof}
By Lemma \ref{lem:rank2 basic properties}, $D$ has $t\ge0$ isolated vertices and a single nontrivial connected component. We may assume henceforth that $t=0$, so that $D$ is connected. By Lemma \ref{lem:complete bi or tripartite rank 2}, $G(D)$ is either a complete bipartite graph $K_{a,b}$ or a complete tripartite graph with parts $A$, $B$, $C$, where $|A|=a$, $|B|=b$ and $|C|=c$. By Lemma \ref{lem:twin reduction}, we may assume that $D$ has no twins, i.e., $D=T(D)$.

Let us first assume that $G(D)=K_{a,b}$, where $1\le a\le b$. If $a=b=1$, then $T(D)$ is switching equivalent to $K_2$, which gives the first outcome. Similarly, if $G(D)=K_{a,b,c}$ is a complete tripartite graph: If $a=b=c=1$, we have the second outcome.
Suppose now that $b>1$. For a vertex $x\in V(D)$, let $H_x$ denote its row in $H(D)$. Let $u$ be a vertex in the first part of the bi- or tripartition and let $v,v'$ be in the second part. Then $H_u$ and $H_v$ are linearly independent, and they generate the row space of $H(D)$. In particular, $H_{v'}=\alpha H_u + \beta H_v$. Note that the support of the vector $H_u$ contains $v$, while the supports of $H_{v'}$ and $H_v$ do not contain $v$. Therefore, $\alpha=0$ and hence $H_{v'}$ is a scalar multiple of $H_v$. In other words, $v'$ is a twin of $v$, a contradiction. This proves the first part of the theorem.

The first equality in (\ref{eq:1}) and (\ref{eq:2}) is obvious. The second one follows from the fact that $\tr(H^2) = 2\rho^2$ is twice the number of edges of $D$ (see \cite{GuoMo}).
The rest of the claims in the theorem are all easy to prove and are left to the reader.
\end{proof}

Let us observe that two mixed digraphs with $H$-rank 2 are cospectral if and only if they have the same number of edges. This will be used in the rest of this section.

As mentioned earlier, balanced complete bipartite graphs $K_{n,n}$ and $K_{n,n+1}$ are DS. However, they are not always DHS. For example, $K_{3,4}$ is cospectral with $\Cabc{2,2,2}$ and $K_{7,8}$ is cospectral with $\Cabc{8,4,2}$. Table \ref{tab:KnnDHS} shows some further examples.

\begin{table}
\label{tab:KnnDHS}
  \centering
\begin{tabular}{|l|l|}
  \hline
  $K_{2,2}$ & DHS \\
  $K_{3,3}$ & DHS \\
  $K_{4,4}$ & $\Cabc{3,2,2}$ \\
  $K_{5,5}$ & DHS \\
  $K_{6,6}$ & $\Cabc{6,3,2},\Cabc{8,2,2}$ \\
  $K_{7,7}$ & DHS \\
  $K_{8,8}$ & $\Cabc{6,4,4}$ \\
  $K_{9,9}$ & $\Cabc{7,6,3}$ \\
  $K_{10,10}$ & $\Cabc{11,6,2}$ \\
  $K_{11,11}$ & $\Cabc{10,7,3}$ \\
  $K_{12,12}$ & $\Cabc{8,8,5}, \Cabc{9,6,6}, \Cabc{12,6,4}, \Cabc{14,6,3}, \Cabc{16,4,4}$ \\
  \hline
\end{tabular}
  \caption{Some complete bipartite graphs and their cospectral mixed graphs}\label{tab:KnnDHS}
\end{table}

As a typical explanation for the entries in Table \ref{tab:KnnDHS}, let us prove that $K_{8,8}$ is cospectral with $\Cabc{6,4,4}$ and, up to switching equivalence, with no other mixed graphs. Other cases in Table \ref{tab:KnnDHS} are treated in the same way. The process shows that the complete tripartite digraphs listed in the table are all cospectral mates of the corresponding complete bipartite graphs (up to switching equivalence).

So, let us consider $K_{8,8}$ which has 64 edges. Since $K_{8,8}$ is DS, it is not cospectral to any other complete bipartite graph. By Theorem \ref{thm:rank 2}, $K_{8,8}$ is not DHS if and only if there are positive integers $a\ge b\ge c$ such that $ab+ac+bc=64$ and $a+b+c\le 16$ (in which case it is cospectral with $\Cabc{a,b,c}\cup (16-a-b-c)K_1$). From these conditions we see that $3c^2\le ab+ac+bc =64$, which yields that $c\le4$. The condition $ab+ac+bc=64$ can be rewritten as
$$
    (a+c)(b+c) = 64 + c^2.
$$
This means that $c^2 + 64$ can be factored in a product $xy$ so that
\begin{equation}
\label{eq:abc for K88}
   x+y = a+b+2c \le 16+c.
\end{equation}
Note that $c^2+64$ equals $65=5\cdot 13$ (for $c=1$), $68=4\cdot 17$ (for $c=2$), $73$ (for $c=3$), and $80=2^4\cdot 5= 5\cdot 16 = 10\cdot 8= 20\cdot 4 = 40\cdot 2$ (for $c=4$). Note that the only factorizations satisfying (\ref{eq:abc for K88}) is $80=10\cdot 8$ for $c=4$. This shows that the only cospectral case is $\Cabc{6,4,4}$.

In conclusion, we shall provide some families of mixed DHS graphs. We first make the following observation. If $K_{m,n}$ is cospectral with $\Cabc{a,b,c}$, then for every integer $t\ge1$, $K_{tm,tn}$ is cospectral with $\Cabc{ta,tb,tc}$. This implies:

\begin{proposition}
There are infinitely many integers $n$ with the property that $K_{n,n}$ is not DHS.
\end{proposition}

This proposition also follows from the following.

\begin{proposition}
If $m$ is divisible by a square $p^2$ and $n>(p-1)m/p^2$, then $K_{m,n}$ is not DHS. More precisely, if $p^2\mid m$, then for any positive integers $a,b$ with $a+b=p$ and $n>abm/p^2$, $K_{m,n}$ is cospectral with $\Cabc{n-t,am/p, bm/p} \cup tK_1$, where $t=abm/p^2$.
\end{proposition}

\begin{proof}
Let us first observe that all parameters are positive integers (where the condition $n>abm/p^2$ is needed to guarantee that $n-t>0$.
Next, we observe that $\Cabc{n-t,am/p, bm/p}$ has $mn$ edges, since
$$
   (n-t)(am/p+bm/p)+(am/p)(bm/p) = (n-t)m+abm^2/p^2 = mn-tm+tm = mn.
$$
Therefore, it has the same nonzero eigenvalues as $K_{m,n}$. Finally, by adding $t$ isolated vertices, the resulting mixed graph has precisely $n$ vertices, thus it is cospectral to $K_{m,n}$.
\end{proof}

The last proposition shows that $K_{n,n}$ is not DHS if $n$ is not square-free.
However, we were not able to decide whether there are infinitely many integers $n$ for which $K_{n,n}$ is DHS. As this seems very unlikely, we dare to conjecture the following.

\begin{conjecture}
There are only finitely many integers $m$ and $n$ for which the complete bipartite graph $K_{m,n}$ is DHS.
\end{conjecture}

To conclude, we observe that digraphs $\Cabc{a,b,c}$ are DHS in many circumstances.

\begin{corollary}
Digraphs $\Cabc{n,n,n}$, $\Cabc{n,n,n+1}$, and $\Cabc{n-1,n,n}$ are DHS for every~$n$.
\end{corollary}

\begin{proof}
By Theorem \ref{thm:rank 2}, $\Cabc{a,b,c}$ ($0\le a\le b\le c$) is DHS if and only if for arbitrary integers $0\le x\le y\le z$ satisfying $xy+xz+yz=ab+ac+bc$ and $x+y+z\le a+b+c$ we have that $(x,y,z) = (a,b,c)$. We include the possibility that the parameters are 0 in order to treat the complete bipartite graphs at the same time and in the same way as the complete tripartite case. For the cases of this corollary we will prove that whenever $x+y+z\le a+b+c$, we have that $xy+xz+yz<ab+ac+bc$ unless $(x,y,z) = (a,b,c)$. To prove this, we may assume that $x+y+z=a+b+c$.

If $x\le z-2$, we may replace $x,y,z$ with $x+1,y,z-1$. This triple increases the mixed sum, an therefore the maximum value of $xy+xz+yz$ is when $x,y,z$ are within one from each other. For each value of $x+y+z$ modulo 3 there is a unique maximal solution that is given by the cases of the corollary. This completes the proof.
\end{proof}

The digraphs $\Cabc{n,n,n}$, $\Cabc{n,n,n+1}$, and $\Cabc{n-1,n,n}$ are DHS because they have smallest number of vertices for the number of edges they contain. There are many other cases, where DHS follows by number theoretic reasons.
One can show also that $\Cabc{n,n,n+2}$ and $\Cabc{n,n,n+3}$ is DHS.
Let us display another interesting family.

\begin{corollary}
\label{cor:big DHS family}
Suppose that $a$ and $n>a>0$ are integers such that $a^2<2n$.  Then the digraph $\Cabc{n-a,n,n+a}$ is DHS if and only if $a$ is not divisible by a prime that is congruent to $1$ modulo~$6$.
\end{corollary}

\begin{proof}
The number of edges of $\Cabc{n-a,n,n+a}$ is $e=3n^2-a^2$. It follows by Theorem \ref{thm:rank 2} that this graph is not DHS if and only if there exist integers $x\le y\le z$, where $(x,y,z)\ne (-a,0,a)$, such that $x+y+z\le0$, $n+x\ge 0$, and
\begin{equation}
   \label{eq:Heronian means 1}
   -a^2 = 2n(x+y+z) + xy + xz + yz.
\end{equation}
If $x,y,z$ exist, then a cospectral mate of $\Cabc{n-a,n,n+a}$ is $\Cabc{n+x,n+y,n+z}$ together with $t=-(x+y+z)$ isolated vertices.

Let us assume that $x,y,z$ exist and let $d=-(x+y+z)\ge0$. According to (\ref{eq:Heronian means 1}, we are interested in the value $\phi(x,y,z)=2nd-(xy+xz+yz)$. We claim that $\phi(x,y,z)>a^2$ for every instance where $d\ge1$.
Let us observe that $d^2=x^2+y^2+z^2+2(xy+xz+yz)$. If $d=1$, this implies that $2(xy+xz+yz)=1-(x^2+y^2+z^2)\le0$, since the left side is even and thus $x,y,z$ cannot all be equal to 0. Consequently, $\phi(x,y,z)\ge 2n > a^2$. If $d>1$, we use induction and apply the inductive hypothesis to the triple $x+1,y,z$:
$$
    \phi(x,y,z) = \phi(x+1,y,z) + 2n + y + z > a^2 + 2n + y + z.
$$
It suffices to see that $2n+y+z\ge0$, which follows from the fact that $z\ge y\ge x\ge -n$. This proves the claim and shows that no solution exists for $d\ge1$.

It remains to treat the case when $d=0$. Since $(x,y,z)\ne (-a,0,a)$, (\ref{eq:Heronian means 1}) implies that $y\ne0$. If $y>0$, let $p=y$ and $q=z$. Otherwise, let $p=-x$ and $q=-y$. The necessary condition (\ref{eq:Heronian means 1}) reduces to the condition $a^2=xy+xz+yz$, which can be written in terms of $p$ and $q$ as:
\begin{equation}
   \label{eq:Heronian means 2}
   a^2 = p^2 + q^2 +pq.
\end{equation}
Therefore $a,p,q$ are integers which occur as the side lengths of some triangle with integer sides and a 120 degree angle. The integers $a$ for which a solution exists are found in the On-Line Encyclopedia of Integer Sequences \cite{EncyclopediaSequences} as the Sequence A050931. It appears that the elements of this sequence are precisely all multiples of primes that are congruent to 1 modulo 6.
\end{proof}

\bibliographystyle{plain}
\bibliography{spectradigraphs}

\end{document}